\documentclass{amsart}

\usepackage{amsfonts,amssymb,amsmath,amscd}

\theoremstyle{plain}
\newtheorem{thm}{Theorem}[section]
\newtheorem{lemma}[thm]{Lemma}
\newtheorem{prop}[thm]{Proposition}

\newtheorem{cor}[thm]{Corollary}

\theoremstyle{remark}

\theoremstyle{definition}
\newtheorem{defn}[thm]{Definition}
\newtheorem{example}[thm]{Example}

\newcommand{\End}{\operatorname{End}}

\newcommand{\Aut}{\operatorname{Aut}}

\begin{document}
\title{Duals of simple two-sided vector spaces}
\author{J. Hart}
\address{Department of Mathematics, University of Montana, Missoula, MT 59812-0864}
\author{A. Nyman}
\address{Department of Mathematics, 516 High St, Western Washington University, Bellingham, WA 98225-9063}
\email{adam.nyman@wwu.edu}
\date{\today}
\thanks{2010 {\it Mathematics Subject Classification. } Primary 16D20 ; Secondary 16S38}
\thanks{The second author was partially supported by the National Security Agency under grant NSA H98230-05-1-0021.}

\begin{abstract}
Let $K$ be a perfect field and let $k \subset K$ be a subfield.  In \cite[Theorem 3.2]{papp},
left finite dimensional simple two-sided $k$-central vector spaces over $K$ were classified by arithmetic data
associated to the extension $K/k$.  In this paper, we continue to study the relationship between simple two-sided vector spaces
and their associated arithmetic data.  In particular, we determine which arithmetic data corresponds to simple two-sided vector spaces with the same left and right dimension, and
we determine the arithmetic data associated to the left and right dual of a simple two-sided vector space.
As an immediate application, we prove the existence of the non-commutative symmetric algebra of any $k$-central two-sided
vector space over $K$ which has the same left and right dimension.
\end{abstract}

\maketitle

\section{Introduction}
In non-commutative algebraic geometry, one thinks of graded
non-commutative rings as homogeneous coordinate rings of
non-commutative spaces. If a graded non-commutative ring has certain
features in common with a commutative ring $A$, the associated
space is considered a non-commutative analogue of
$\operatorname{Proj }A$.  A standard approach to constructing
non-commutative analogues of the projective line over a field $K$
is to take its homogeneous coordinate ring to be a quotient of the
free $K$-algebra $K \langle x, y \rangle$ having favorable
homological properties.  For example, if $K$ is algebraically
closed and $B=K \langle x, y \rangle/ (yx-qxy)$ for some $q \in
K^{*}$ or $B=K \langle x, y \rangle/ (yx-xy-x^{2})$, then $B$ is
Artin-Schelter regular of global dimension 2 \cite[Proposition
2.1, Chapter 17]{paul}, and $B$ is considered a homogeneous
coordinate ring for a non-commutative analogue of
$\mathbb{P}^{1}$.  Note that in these cases, $B$ is
obtained from $A=K[x,y]$ by deforming the commutativity relation
$yx-xy$.

If $W$ is a two-dimensional vector space over $K$ and
$\mathcal{S}_{K}(W)$ denotes the symmetric algebra of
$W$, then there is an isomorphism of graded $K$-algebras
$K[x,y] \rightarrow \mathcal{S}_{K}(W)$.  This suggests another
way to construct a homogeneous coordinate ring of a non-commutative
projective line, due to M. Van den Bergh \cite[Section 1]{p1bundle}.  Let $k$ be a subfield of $K$.  Instead of deforming
relations in $K[x,y]$, one deforms $W$ by replacing it with a $k$-central
{\it two-sided} vector space $V$ of rank two, i.e. $V$ is a $K \otimes_{k}K$-module which
is two-dimensional as both a $K \otimes 1$-module and a $1 \otimes K$-module.
Van den Bergh shows that when $K$ is the field of fractions of a smooth integral scheme $X$ of finite type over $k$ and $V$ is the generic localization of
a locally free rank 2 $\mathcal{O}_{X}$-bimodule (see \cite{p1bundle} for a definition), one can
construct a canonical non-commutative ring from $V$,
$\mathcal{S}_{K}^{n.c.}(V)$, in such a way that
$\mathcal{S}_{K}^{n.c.}(V)$ is ``equivalent" to
$\mathcal{S}_{K}(V)$ when $V$ is $K$-central \cite[Section 5.2]{p1bundle}.

Now suppose $K$ is a perfect field and let $k \subset K$ be a subfield.  In \cite[Theorem 3.2]{papp},
simple two-sided $k$-central vector spaces over $K$ are classified by arithmetic data
associated to the extension $K/k$.  In this paper, we continue to study the relationship between the arithmetic of the extension $K/k$ and simple two-sided vector spaces.  As an immediate application of our study, we prove the existence of non-commutative symmetric algebras.  Our initial motivation for studying non-commutative symmetric algebras comes from a conjecture of Mike Artin \cite{artin} that
the division ring of fractions of a non-commutative surface not finite over its center is the division ring of fractions of a non-commutative symmetric algebra over the field of fractions of a curve.
The investigations in this paper are part of an ongoing project to classify division rings of fractions of non-commutative symmetric algebras of rank 2 two-sided vector spaces.

Our first goal is to find an arithmetic criterion for a simple two-sided vector space with left dimension $n$ to have rank $n$, i.e. to have right dimension also equal to $n$.
We achieve this goal in Proposition \ref{prop.vlambda}.  In order to describe this result, we need to recall the classification of simple two-sided vector spaces of finite left dimension.  To this end, we introduce some notation.
We write $\operatorname{Emb}(K)$ for the set of $k$-linear embeddings of $K$ into $\overline{K}$, and we let $G=\Aut(\overline{K}/K)$.  Now, $G$ acts on $\operatorname{Emb}(K)$ by left composition. Given $\lambda\in \operatorname{Emb}(K)$, we denote the orbit of $\lambda$ under this action by $\lambda^G$.  We denote the set of finite orbits of $\operatorname{Emb}(K)$ under the action
of $G$ by $\Lambda(K)$.  Left finite dimensional simple two-sided vector spaces are classified by the following

\begin{thm} \label{thm.papp} \cite{papp} There is a one-to-one correspondence
between isomorphism classes of simple left finite dimensional
two-sided vector spaces and $\Lambda(K)$. Moreover, if $V$ is a
simple two-sided vector space corresponding to $\lambda^G\in
\Lambda(K)$, then $\dim_{K} ({}_{K}V) =|\lambda^G|$.
\end{thm}
Let $\lambda \in \operatorname{Emb}(K)$ have $|\lambda^{G}|=n$ and let $K(\lambda)$ denote the composite of $K$ and $\lambda(K)$.  We define
a two-sided vector space, $V(\lambda)$, as follows:  let $V(\lambda)$ have underlying set $K(\lambda)$ and
$K \otimes_{k} K$-structure induced by the formula
$a \cdot v \cdot b := av\lambda(b)$.

We observe the following (Proposition \ref{prop.vlambda})
\begin{prop} \label{prop.introv}
Suppose $V$ is a left finite dimensional simple
two-sided vector space (of left dimension $n$)
corresponding (via Theorem \ref{thm.papp}) to an embedding $\lambda$.  Then $V \cong V(\lambda)$.  Therefore,
$V$ has rank $n$ if and only if $[K(\lambda):\lambda(K)]=n$.
\end{prop}

This criterion, apart from being useful in the study of moduli of two-sided vector spaces, facilitates the explicit construction of
simple two-sided vector spaces of finite rank.  For example, in Section \ref{sec.family}, we use Proposition \ref{prop.introv} in the case that $K=k(t)$ where $t$ is transcendental over $k$
to exhibit a five-dimensional family of simple two-sided vector spaces $V$ over $k(t)$ of rank 2.  More specifically, we obtain the following
\begin{prop} \label{prop.john}
Suppose $\operatorname{char }k =0$.  If $V$ is a simple two-sided vector-space over $k(t)$ corresponding (via Theorem \ref{thm.papp}) to $\lambda \in \operatorname{Emb}(k(t))$ with $\lambda(t)=\alpha+\sqrt{\frac{at^2+bt+c}{dt^2+et+f}}$ such that
\begin{itemize}
\item $\alpha, a, b, c, d, e, f \in k$,

\item $a$, $d$ not both zero,

\item $ae=bd$, $af \neq cd$, $b^2 \neq 4ac$, and $e^2 \neq 4df$,
\end{itemize}
then $V$ has rank 2.
\end{prop}
In Section \ref{sec.duals}, we turn to the study of left and right duals of two-sided vector spaces, since they arise in the construction of the non-commutative symmetric algebra.
We find formulas for the left and right dual of a simple two-sided
vector space of finite rank in terms of its associated arithmetic data.  Our main result is the following (Theorem \ref{thm.dualformula})

\begin{thm} \label{thm.introthm}
Suppose $\lambda \in \operatorname{Emb}(K)$ has $|\lambda^{G}|=n$ and $[K(\lambda):\lambda(K)]$ is finite.  Let $\overline{\lambda}$ denote an extension of $\lambda$ to $\overline{K}$, and let
$\mu = (\overline{\lambda})^{-1}|_{K}$.  Then ${}^{*}V(\lambda) \cong V(\lambda)^{*} \cong V(\mu)$, and $V(\lambda)$ has rank $n$ if and only if $V(\mu)$ has rank $n$ if and
only if $[K(\lambda):\lambda(K)]=n$.
\end{thm}
In Section \ref{sec.ncsym}, we recall the definition of the non-commutative symmetric algebra of a rank $n$ two-sided vector space, $V$, over $K$ (from \cite{p1bundle}).  Using Theorem \ref{thm.introthm}, we prove it exists if $K$ is perfect (Corollary \ref{cor.ncsym}).  Existence of the non-commutative symmetric algebra amounts to showing that the iterated duals of $V$ have rank $n$.

D. Patrick also studied the question of when a non-commutative symmetric algebra of $V$ exists in the rank 2 case \cite{P}.  He then analyzed the structure of the algebra (when $V$ is not simple) and computed its field of fractions.  His notion of non-commutative symmetric algebra predates (and is distinct from) the one used in this paper.  For the precise relationship between the two notions, see \cite[Section 5.2]{p1bundle}.

Some of the results in this paper are part of the first author's Ph.D. thesis \cite{john} written under the supervision of the second author.

{\it Acknowledgement: } We thank D. Chan and S.P. Smith for interesting conversations which led to improvements in this paper and we thank Nikolaus Vonessen for carefully reading an earlier version of this paper and making numerous suggestions for its improvement.

\section{Simple two-sided vector spaces} \label{sec.twosided}

\subsection{Preliminaries on two-sided vector spaces}
Throughout the paper, we let $k \subset K$ be a field extension, and we let $\bar K$ denotes
a fixed algebraic closure of $K$.  By a \emph{two-sided vector
space} we mean a $K \otimes_{k}K$-module $V$.  By {\it right (resp. left) multiplication by $K$} we mean multiplication by elements in $1 \otimes_{k} K$ (resp. $K \otimes_{k} 1$).  We denote the restriction of scalars of $V$ to $K \otimes_{k} 1$ (resp. $1\otimes_{k} K$) by ${}_{K}V$ (resp. $V_{K}$).

If $V$ is a two-sided vector space, then right multiplication by
$x\in K$ defines an endomorphism $\phi(x)$ of $_KV$, and the right
action of $K$ on $V$ is via the $k$-algebra homomorphism
$\phi:K\rightarrow \End(_KV)$.  This observation motivates the
following

\begin{defn} Let $\phi:K\rightarrow M_n(K)$ be a nonzero homomorphism.
Then we denote by $K^n_\phi$ the two-sided vector space of left
dimension $n$, where the left action is the usual one and the
right action is via $\phi$;
 that is,
\begin{equation}x\cdot(v_1,\dots, v_n)=(xv_1,\dots,xv_n),
\ \ \ (v_1,\dots, v_n)\cdot
x=(v_1,\dots,v_n)\phi(x).\end{equation} We shall always write
scalars as acting to the left of elements of $K^n_\phi$ and
matrices acting to the right; thus, elements of $K^n$ are written
as row vectors.
\end{defn}

If $V$ is a two-sided vector space with left dimension equal to
$n$, then choosing a left basis for $V$ shows that $V\cong
{K^n_\phi}$ for some homomorphism $\phi:K\rightarrow M_n(K)$.

We say $V$ has {\it rank $n$} if $\operatorname{dim}_{K}({}_{K}V)
= \operatorname{dim}_{K}(V_{K})=n$.  If $V$ has rank $n$, then $V$ has a
{\it simultaneous basis}, i.e. a subset $\{y_1,\ldots,y_n\}$
which is a basis for both the left and right action of $K$
on $V$ \cite[p. 18]{P}.

\subsection{Simple two-sided vector spaces and their left and right dimension} \label{sect.babysimple}
In this section, we assume that $K$ is a perfect field in order to employ Theorem \ref{thm.papp}.

We first obtain a convenient description of the simple two-sided vector space
associated to $\lambda \in \operatorname{Emb}(K)$ with $|\lambda^{G}|=n$.
We recall (from \cite[Section 3]{papp}) that the simple two-sided vector space associated to such a $\lambda$ is
$K^{n}_{\phi}$, where $\phi:K \rightarrow M_{n}(K)$ is the $k$-algebra homomorphism defined as follows.  Let
\begin{equation} \label{eqn.simpledata}
\begin{split}
\bullet & \{\alpha_{1},\ldots,\alpha_{n}\} \mbox{ be a basis for } K(\lambda)/K  \\
\bullet & \lambda_{i}:K \rightarrow K \mbox{ be defined by } \lambda(x) = \mbox{$\sum_i \lambda_{i}(x)\alpha_{i}$}, \mbox{ and } \\
\bullet & \beta_{ijk} \mbox{ be defined by the equation $\alpha_{i}\alpha_{j}=\sum_k \beta_{ijk}\alpha_{k}$}. \\
\end{split}
\end{equation}
We define
\begin{equation} \label{eqn.phi}
\phi_{ij}(x) = \sum_{k=1}^{n}\beta_{jki}\lambda_{k}(x).
\end{equation}

\begin{lemma} \label{lemma.transpose}
Suppose $K^{n}_{\phi}$ is the simple module corresponding to $\lambda$.  Then $K^{n}_{\phi^{T}}\cong K^{n}_{\phi}$.
\end{lemma}

\begin{proof}
Suppose $K^{n}_{\phi}$ corresponds to $\lambda \in \operatorname{Emb}(K)$, and $\{\sigma_{1}\lambda,\ldots,\sigma_{n}\lambda\}=\lambda^{G}$.  Then there exists $A \in \operatorname{GL}_{n}(\overline{K})$ such that $A \phi A^{-1} = D := \operatorname{diag }(\sigma_{1}\lambda,\ldots,\sigma_{n}\lambda)$ \cite[Part 1, Step 1 of proof of Proposition 3.5]{papp}.  Therefore, $D=D^{T}=(A^{-1})^{T}\phi^{T}A^{T}$.  It follows that $\phi$ is similar to $\phi^{T}$ over $\overline{K}$ so that $\phi$ is similar to $\phi^{T}$ over $K$.  The result follows.
\end{proof}

\begin{prop} \label{prop.vlambda}
 The two sided vector space $V(\lambda)$ (defined in the introduction) is simple and corresponds (via Theorem \ref{thm.papp}) to $\lambda$.  Therefore, the left dimension of $V(\lambda)$ over $K$ is $[K(\lambda):K]$, while the right dimension of $V(\lambda)$ over $K$ is $[K(\lambda):\lambda(K)]$.
\end{prop}

\begin{proof}
We show that $V(\lambda) \cong K^{n}_{\phi^{T}}$.  The result will then follow from Lemma \ref{lemma.transpose}.

We need only compute the matrix for the right action of $K$ on the left basis $\{\alpha_{1},\ldots,\alpha_{n}\}$ of $V(\lambda)$.  We have
$$
\alpha_{j} \cdot x := \alpha_{j} \lambda(x) = \sum_{i=1}^{n}\alpha_{j}\lambda_{i}(x)\alpha_{i} = \sum_{i=1}^{n} \lambda_{i}(x)(\sum_{k=1}^{n} \beta_{ijk}\alpha_{k}) = \sum_{k=1}^{n}(\sum_{i=1}^{n}\lambda_{i}(x)\beta_{ijk})\alpha_{k}.
$$
Thus, in the left-basis $\{\alpha_{1},\ldots,\alpha_{n}\}$, the $jk$ entry of the right action matrix is $\sum_{i=1}^{n}\lambda_{i}\beta_{ijk}$ and therefore, the matrix has $ij$-entry $\sum_{k=1}^{n}\lambda_{k}\beta_{kij}=\sum_{k=1}^{n}\beta_{ikj}\lambda_{k}$.  This is just the $ij$-entry of $\phi^{T}$, so that $V(\lambda) \cong K^{n}_{\phi^{T}}$.
\end{proof}

\subsection{A five-dimensional family of simple two-sided vector spaces} \label{sec.family}
For the remainder of Section \ref{sec.twosided}, we assume $\operatorname{char }k=0$ and $t$ is transcendental over $k$.

\begin{lemma} \label{lemma.numerical}
Let $m=\frac{at^2+bt+c}{dt^2+et+f}$ with $a,b,c,d,e,f \in k$, $a,d$ not both zero, and $ae=bd$, $af \neq cd$, $b^2 \neq 4ac$, and $e^2 \neq 4df$.  Then $[k(t):k(m)]=2$ and $\sqrt{m} \notin k(t)$.
\end{lemma}

\begin{proof}
If $m$ were in $k$, then $a=dm$, $c=fm$, so $af=dmf=cd$, a contradiction.  Hence, $m \notin k$.

If $\operatorname{gcd}(at^2+bt+c,dt^2+et+f) \neq 1$, then these two polynomials have a common root $u$ in some extension field of $k$.  So
$$
au^2+bu+c=0=du^2+eu+f.
$$
Hence
$$
adu^2+bdu+cd=0=adu^2+aeu+af.
$$
Since $bd=ae$, it follows that $af=cd$, a contradiction.  Since $a$ or $d$ is nonzero, it follows that $[k(t):k(m)]=2$.

Now suppose $\sqrt{m} \in k(t)$.  Then there exist polynomials $p,q \in k[t]$ such that $m=\frac{p^2}{q^2}$. We may assume $\operatorname{gcd}(p,q)=1$.  Now
\begin{equation} \label{eqn.star}
p^2(dt^2+et+f)=q^2(at^2+bt+c).
\end{equation}
Since $\operatorname{gcd}(p,q)=1$,
$$
p^2|at^2+bt+c, q^2|dt^2+et+f.
$$
Hence, $\operatorname{deg }p,q \leq 1$.

Suppose $a=0$.  Then $d \neq 0$, so that the left hand side of (\ref{eqn.star}) has even degree.  Hence the right hand side of (\ref{eqn.star}) has even degree, implying $b=0$.  Hence, $b^2=0=4ac$, a contradiction.  Thus $a \neq 0$.  Similarly, if $d=0$ then $e=0$, so that $e^2=0=4df$, a contradiction.  Consequently both $a \neq 0$ and $d \neq 0$.

It now follows from (\ref{eqn.star}) that $\operatorname{deg }p=\operatorname{deg }q$.  If $\operatorname{deg }p=\operatorname{deg }q=0$, then $m=\frac{p^2}{q^2} \in k$, a contradiction.  Hence, $\operatorname{deg }p=\operatorname{deg }q=1$.  Now $at^2+bt+c=u_{1}p^2$ for some $u_{1} \in k$.  If $t+x$ ($x \in k$) is the monic polynomial associated to $p$, then
$$
at^2+bt+c=u_{2}(t+x)^2,
$$
for some $u_{2} \in k$.  Clearly, $u_{2}=a$.  We obtain
$$
at^2+bt+c=at^2+2axt+ax^2.
$$
Hence $b=2ax$, $c=ax^2$, and $b^2=4a(ax^2)=4ac$, a contradiction.  This final contradiction concludes the proof.
\end{proof}

We now prove Proposition \ref{prop.john}.  Retain the notation from Lemma \ref{lemma.numerical}.
By Lemma \ref{lemma.numerical}, $[k(t):k(m)]=2$, so that $[k(t,\sqrt{m}):k(\sqrt{m})]=2$.  Since
$\alpha \in k$, we may conclude that $[k(t,\sqrt{m}):k(\alpha+\sqrt{m})]=2$.  Therefore, $[K(\lambda):\lambda(K)]=2$.
On the other hand, Lemma \ref{lemma.numerical} implies that $[k(t,\sqrt{m}):k(t)]=2$ so that $[K(\lambda):K]=2$.
Therefore, the simple two-sided vector space corresponding to $\lambda$ has rank 2.

\section{Duals of two-sided vector spaces} \label{sec.duals}
In order to define the non-commutative symmetric algebra of a two-sided vector space, we need to recall
(from \cite[Section 4]{p1bundle}) and study the
notion of the left and right dual of a two-sided vector space.  Our main goal in this section is to determine a
formula for the dual of a simple two-sided vector space.  This goal is realized in Theorem \ref{thm.dualformula}.

For any commutative ring $R$, we let ${\sf Mod }R$ denote the category of $R$-modules.  We let ${\sf R}$ (resp. ${\sf L}$) denote the full subcategory of
${\sf Mod }K \otimes_{k}K$ consisting of modules which are
finite-dimensional over $K$ on the right (resp. left).

\subsection{Duals and their basic properties}
\begin{defn} \label{defn.rightdual}
The {\it right dual of $V$}, denoted $V^{*}$, is the set
$\operatorname{Hom}_{K}(V_{K},K)$ with action $(a \cdot \psi \cdot
b)(x)=a\psi(bx)$ for all $\psi \in
\operatorname{Hom}_{K}(V_{K},K)$ and $a,b \in K$.  We note that
$V^{*}$ is a $K\otimes_{k}K$-module since $V$ is.

The {\it left dual of $V$}, denoted ${}^{*}V$, is the set
$\operatorname{Hom}_{K}({}_{K}V,K)$ with action $(a \cdot \phi
\cdot b)(x)=b \phi(xa)$ for all $\phi \in
\operatorname{Hom}_{K}({}_{K}V,K)$ and $a,b \in K$.  As above,
${}^{*}V$ is a $K \otimes_{k}K$-module.

We set
$$
V^{i*}:=
\begin{cases}
V & \text{if $i=0$}, \\
(V^{(i-1)*})^{*} & \text{ if $i>0$}, \\
{}^{*}(V^{(i+1)*}) & \text{ if $i<0$}.
\end{cases}
$$
\end{defn}
We may also form the usual dual of a two-sided vector space $V$, defined by $\check{V} = \operatorname{Hom}_{K \otimes_{k}K}(V,K \otimes_{k}K)$ with its usual $K\otimes_{k}K$-module structure.  The following example shows that it is not always true that $\check{V} \cong V^{*}$.
\begin{example}
Let $x,y$ be algebraically independent transcendentals over $k$, and let $K=k(x)$.  Then $K \otimes_{k}K$ is a domain since it is a subring of $k(x,y)$.  In addition, let $\sigma$ denote a nontrivial $k$-linear automorphism of $K$ and suppose $c \in K$ is not a fixed point.

We define a two-sided vector space $V$ with underlying set $K$ and $K\otimes_{k}K$-action
induced by the formula $(a\otimes b) \cdot x := a \sigma(b) x$.

Since $V$ is simple, $\check{V}$ is either equal to $0$, or $V$ embeds in $K \otimes_{k}K$.
 But $K \otimes_{k}K$ is a domain, and for $v \in V$, $(1\otimes c - \sigma(c) \otimes 1) \cdot v = 0$.
  Therefore, $V$ cannot embed in $K \otimes_{k}K$ as a $K \otimes_{k}K$-module.  The contradiction establishes the fact that $\check{V} = 0$.  On the other hand, we leave it as an exercise to check that $V^{*}$ is isomorphic to the $K \otimes_{k}K$-module with underlying set $K$ and action induced by the formula $(a \otimes b) \cdot x := a \sigma^{-1}(b) x$.
\end{example}
It is straightforward to check that the assignments on objects $(-)^{*}:{\sf Mod }K \otimes_{k}K \rightarrow {\sf
Mod }K \otimes_{k}K$ and ${}^{*}(-):{\sf Mod }K \otimes_{k}K
\rightarrow {\sf Mod }K \otimes_{k}K$ canonically induce contravariant
left exact functors.

\begin{lemma} \label{lemma.functorial}
The functors $(-)^{*}$ and ${}^{*}(-)$ restrict to exact functors
$(-)^{*}:{\sf R} \rightarrow {\sf L}$ and ${}^{*}(-):{\sf L}
\rightarrow {\sf R}$.
\end{lemma}

\begin{proof}
We prove the lemma for $(-)^{*}$. The proof for ${}^{*}(-)$ is
similar and omitted.  Since $V \in {\sf R}$ implies that $\operatorname{dim
}_{K}({}_{K}(V^{*}))=\operatorname{dim }_{K}(V_{K})$ is finite, $(-)^{*}$ restricted to ${\sf R}$ takes
values in ${\sf L}$.

To prove the restriction of $(-)^{*}$ to ${\sf R}$ is exact, it suffices to show that the functor $\operatorname{Hom}_{K}(-,K):{\sf R} \rightarrow {\sf Mod }K$ taking an object $V$ to the usual dual of $V_{K}$ is exact.  This follows from the exactness of the restriction of scalars functor sending an object $V$ in ${\sf R}$ to $V_{K}$ in ${\sf Mod }K$.
\end{proof}

\subsection{Adjoint pairs from duals}
As one might expect, if $V$ has rank $n$ then both pairs of functors $(-\otimes_{K}{}^{*}V,-\otimes_{K}V)$ and $(-\otimes_{K}V,-\otimes_{K}V^{*})$ from ${\sf Mod }K$ to ${\sf Mod }K$ are adjoint pairs.  This is mentioned in \cite[Section 4]{p1bundle} and left as an exercise for the reader.  For completeness, and in order to spare the reader from verifying that the proof is independent of the geometric assumptions made in \cite{p1bundle}, we include it here (in Proposition \ref{prop.adjoint}).  We note that the conclusion of Proposition \ref{prop.adjoint} remains true in case $V$ has distinct finite left and right dimension.  We leave the verification of this fact to the interested reader.

In our proof of Proposition \ref{prop.adjoint} and related facts, we will use the following notation for the rest of the section.  Let $F$ denote the ring of additive functions from $K$ to $K$
under composition and point-wise addition and let $M_{n}(F)$ denote
the corresponding matrix ring.  Let $V$ be a two-sided vector space of rank $n$, and let $y_{1},\ldots, y_{n}$ be a simultaneous basis for $V$.
Following \cite[Section 1]{P}, we define a matrix $A
\in M_n(F)$ as follows.  For $\alpha \in K$,
$$
y_{i}\alpha = \sum_{j}a_{ij}(\alpha)y_{j}.
$$
We call $A$ the {\it matrix representing right multiplication on $V$ in the
left basis $\{y_i\}$}.

Similarly, we define a matrix $B \in M_n(F)$ as follows.  For
$\delta \in K$,
$$
\delta y_{i}= \sum_j y_{j}b_{ji}(\delta).
$$
We call $B$ the {\it matrix representing left multiplication on
$V$ in the right basis $\{y_{i}\}$}.

Let $\{\psi_{i}\} \subset V^{*}$ be defined by
$\psi_{i}(y_{j})=\delta_{ij}$, and let $\{\phi_{i} \} \subset {}^{*}V$
be defined similarly.

\begin{lemma} \label{lemma.matrix}
The matrix representing right multiplication on $V^{*}$ in the left basis $\{\psi_{i}\}$ is
$B$ and the matrix representing left multiplication on ${}^{*}V$ in the right basis
$\{\phi_{i}\}$ is $A$.
\end{lemma}

\begin{proof}
Suppose $a_{1},\ldots,a_{n}, \alpha, \delta$ are in $K$.  We
compute
$$
\psi_{i} \cdot \delta (\sum_j y_{j}a_{j})  =  \psi_{i}(\sum_j \delta y_{j}a_{j})  =  \psi_{i}(\sum_j(\sum_k y_{k} b_{kj}(\delta)a_{j})) = \sum_j b_{ij}(\delta)a_{j}.
$$
Since the last expression equals $\sum_j b_{ij}(\delta)\psi_{j}(\sum_l y_{l}a_{l})$, we conclude that $\psi_{i} \cdot \delta = \sum_j b_{ij}(\delta)\psi_{j}$.
A similar computation establishes the equality $\alpha \cdot \phi_{i} = \sum_j \phi_{j} a_{ji}(\alpha)$.
\end{proof}

\begin{cor} \label{cor.matrix0}
There are isomorphisms $({}^{*}V)^{*} \cong V$ and
${}^{*}(V^{*}) \cong V$.  Therefore, if $V^{i*}$ has rank $n$, then $(V^{i*})^{*} \cong V^{(i+1)*}$ and ${}^{*}(V^{i*}) \cong V^{(i-1)*}$.
\end{cor}

\begin{proof}
We prove the first statement and leave the second as an exercise.  We define $\Phi:V \rightarrow ({}^{*}V)^{*}$ by letting $\Phi(\sum_{i}a_{i}y_{i})(\sum_{j}\phi_{j}b_{j})=\sum_{k}a_{k}b_{k}$, where
$a_{i}$ and $b_{j}$ are in $K$ for all $i$, $j$.  Note that this formula makes sense since $\{\phi_{i}\}$ is a right basis for ${}^{*}V$.  It is elementary to check that $\Phi$ is additive and
compatible with the left $K$-action.  We show that it is
compatible with the right action.  Suppose $\alpha \in K$.  On the one hand,
\begin{eqnarray*}
\Phi(\sum_{i}a_{i}y_{i}\alpha)(\sum_{k}\phi_{k}b_{k}) & = & \Phi(\sum_{i,j}a_{i}a_{ij}(\alpha)y_{j})(\sum_{k}\phi_{k}b_{k}) \\
& = & \sum_{i,j}a_{i}a_{ij}(\alpha)b_{j}.
\end{eqnarray*}
On the other hand
\begin{eqnarray*}
(\Phi(\sum_{i}a_{i}y_{i})\alpha)(\sum_{k}\phi_{k}b_{k}) & = & \Phi(\sum_{i}a_{i}y_{i})(\sum_{k}\alpha \phi_{k}b_{k}) \\
& = & \Phi(\sum_{i}a_{i}y_{i})(\sum_{k}(\sum_{j} \phi_{j}a_{jk}(\alpha))b_{k}) \\
& = & \sum_{k,j}a_{jk}(\alpha)b_{k}a_{j}.
\end{eqnarray*}
It follows that $\Phi$ is a homomorphism of two-sided vector spaces.  Since $\Phi$ is clearly injective, it remains to check that $\Phi$ is surjective.  To this end, since $\{\phi_{i}\}$ is a right
basis for ${}^{*}V$, its right duals in $({}^{*}V)^{*}$ are a left basis for $({}^{*}V)^{*}$.  Therefore, $({}^{*}V)^{*}$ has left dimension $n$.  It follows that $\Phi$ is surjective.
\end{proof}

\begin{prop} \label{prop.dualmatrix}
The matrices $A^{T}$ and $B$ are inverses in the ring $M_{n}(F)$.
\end{prop}

\begin{proof}
For all $\delta$ and $\lambda$ in $K$, we have
$$
(\delta y_{i})\lambda = \sum_j y_{j} b_{ji}(\delta)\lambda  =  \sum_j (\sum_k a_{jk}(b_{ji}(\delta) \lambda) y_{k}) = \sum_k (\sum_j a_{jk}(b_{ji}(\delta) \lambda)) y_{k}.
$$
On the other hand,
$$
\delta (y_{i} \lambda) = \delta (\sum_j a_{ij}(\lambda) y_{j}) =  \sum_k \delta a_{ik} (\lambda) y_{k}.
$$
Therefore, $\delta a_{ik}(\lambda)=\sum_j a_{jk}(b_{ji}(\delta)
\lambda)$.  In particular, when $\lambda = 1$, we have
$$
\delta \delta_{ik} = \sum_j a_{jk}(b_{ji}(\delta)),
$$
where $\delta_{ik}=0$ if $i \neq k$ and $\delta_{ik}=1$ if $i=k$.  The left-hand side is the $(k,i)$-entry of $\delta I_{n}$,
while the right-hand side is the $(k,i)$-entry of
$A^{T}B(\delta)$. Therefore, $A^{T}B=I_{n}$.

We now repeat the computation above, but group terms on the right
of the $y_{i}$'s.  On the one hand, we have
\begin{eqnarray*}
(\delta y_{i}) \lambda & = & \sum_j y_{j}b_{ji}(\delta) \lambda.
\end{eqnarray*}
On the other hand,
$$
\delta(y_{i}\lambda) = \sum_j \delta a_{ij}(\lambda) y_{j} = \sum_j(\sum_k y_{k} b_{kj} (\delta a_{ij}(\lambda))) = \sum_k y_{k}(\sum_j b_{kj}(\delta a_{ij}(\lambda))).
$$
Therefore, $\sum_j b_{kj}(\delta a_{ij}(\lambda)) =
b_{ki}(\delta)\lambda$.  In particular, if $\delta = 1$, we have
$$
\sum_j b_{kj}(a_{ij}(\lambda)) = \delta_{ki}\lambda.
$$
The left-hand side is the $(k,i)$-entry of $BA^{T}(\lambda)$ while
the right-hand side is the $(k,i)$-entry of $\lambda I_{n}$.
Therefore, $BA^{T}=I_{n}$.
\end{proof}

\begin{prop} \label{prop.adjoint}
Suppose $V$ has rank $n$.  Then there exist unit and counit morphisms making the pair of functors $(-\otimes_{K}{}^{*}V,-\otimes_{K}V)$
(resp. $(-\otimes_{K}V,-\otimes_{K} V^{*})$) from ${\sf Mod }K$ to ${\sf Mod }K$ adjoint.
\end{prop}

\begin{proof}
We prove the first assertion.  The proof of the second is similar and we omit
it.  Throughout the proof, unlabeled tensor products are over $K$.  We begin by defining two maps $\eta:K \rightarrow {}^{*}V \otimes V$ and
$\epsilon:V \otimes {}^{*}V \rightarrow K$
by $\eta(a)= a \sum_i \phi_{i} \otimes y_{i}$ and
$\epsilon(\sum_{i,j}a_{i}y_{i} \otimes \phi_{j}b_{j}) = \sum_i
a_{i}b_{i}$.  We first show that $\eta$ is a $K
\otimes_{k}K$-module morphism.  To this end, we note that $\eta(ab)$ equals
$$
a(\sum_i b\phi_{i} \otimes y_{i})  = a(\sum_{i,j} \phi_{j}a_{ji}(b) \otimes y_{i}) =  a(\sum_{i,j}  \phi_{j} \otimes a_{ji}(b)y_{i}).
$$
The last expression equals
\begin{equation} \label{eqn.c}
a(\sum_{i,j,k} \phi_{j}\otimes y_{k}b_{ki}(a_{ji}(b))).
\end{equation}
By Proposition \ref{prop.dualmatrix}, $A^{T}=B^{-1}$, which implies that the expression (\ref{eqn.c}) equals $a(\sum_{i} \phi_{i}\otimes y_{i}b)=\eta(a)b$.  It follows that $\eta$ is a $K\otimes_{k}K$-module map.

A routine argument establishes the fact that $\epsilon$ is a $K \otimes_{k}K$-module
map.

Next, a routine computation shows that the compositions
$$
{}^{*}V \cong K \otimes {}^{*}V \overset{\eta \otimes
{}^{*}V}{\longrightarrow} {}^{*}V \otimes V \otimes {}^{*}V
\overset{{}^{*}V \otimes \epsilon}{\longrightarrow} {}^{*}V
\otimes K \cong {}^{*}V
$$
and
$$
V \cong V \otimes K \overset{V \otimes \eta}{\longrightarrow} V
\otimes {}^{*}V \otimes V \overset{\epsilon \otimes
V}{\longrightarrow} K \otimes V \cong V
$$
whose unlabeled maps are canonical, are both identity maps.  It
follows from this and from the Eilenberg-Watts Theorem that the natural transformation $\operatorname{id} \longrightarrow (-\otimes {}^{*}V) \otimes V$ defined by the composition
$$
M \longrightarrow M \otimes K \overset{M \otimes
\eta}{\longrightarrow} M \otimes ({}^{*}V \otimes V)
\longrightarrow (M \otimes {}^{*}V) \otimes V
$$
whose first arrow is canonical and whose last arrow is the
associativity isomorphism, and the natural transformation $(-\otimes V) \otimes {}^{*}V \longrightarrow \operatorname{id}$ defined by the composition
$$
(M \otimes V) \otimes {}^{*}V \longrightarrow M \otimes (V
\otimes {}^{*}V) \overset{M \otimes \epsilon}{\longrightarrow} M \otimes
K \longrightarrow M
$$
whose first arrow is the associativity isomorphism and whose last
arrow is canonical, define a unit and counit of the pair
$(-\otimes_{K}{}^{*}V,-\otimes_{K}V)$.
\end{proof}

We next prove that the image of $\eta$ defined in the proof of Proposition \ref{prop.adjoint} is independent of choice of simultaneous basis for $V$.  To this end, we will need the following lemma, whose proof we leave as an exercise.
\begin{lemma} \label{lemma.mn}
If $U$ has a right basis $\{u_{i}\}_{i=1}^{m}$ and $V$ has
a simultaneous basis $\{y_{j}\}_{j=1}^{n}$ then $U \otimes_{K}V$
has right basis $\{u_{i} \otimes y_{j}\}_{i,j}$.
\end{lemma}

\begin{prop} \label{prop.nochoice}
The image of $\eta$ in the proof of Proposition \ref{prop.adjoint} is independent of choice of simultaneous basis.
\end{prop}

\begin{proof}
Suppose $\{y_{i}\}$ and $\{z_{i}\}$ are simultaneous bases of $V$ with corresponding left dual bases $\{\phi_{i}\}$ and $\{\gamma_{i}\}$.  By Lemma \ref{lemma.mn},
every element of ${}^{*}V \otimes V$ can be written uniquely in the form $\sum_{i,j}\phi_{i} \otimes c_{ij}y_{j}$ where $c_{ij} \in K$.
Thus, we may define a function of abelian groups $\Psi:{}^{*}V \otimes V \rightarrow \operatorname{Hom }_{\sf Ab}(V,V)$ by the formula
 $\Psi(\sum_{i,j}\phi_{i} \otimes c_{ij}y_{j})(v)=\sum_{i,j}\phi_{i}(v)c_{ij}y_{j}$.  It is easy to check that $\Psi$ is an injective group homomorphism.
  We claim that $\sum_{i}\phi_{i} \otimes y_{i}=\sum_{i}\gamma_{i} \otimes z_{i}$ by showing that $\Psi(\sum_{i}\phi_{i} \otimes y_{i})=\Psi(\sum_{i}\gamma_{i} \otimes z_{i})$.
  The proposition will follow from the claim.
To prove the claim, we note that if $a \in K$ then $\Psi(\sum_{i}\phi_{i} \otimes y_{i})(ay_{l})=ay_{l}$ for each $l$.
Since $\Psi(\sum_{i}\phi_{i} \otimes y_{i})$ is additive, it follows that $\Psi(\sum_{i}\phi_{i} \otimes y_{i})$ is the identity function.
In a similar manner, one shows that $\Psi(\sum_{i}\gamma_{i} \otimes z_{i})$ is the identity function.  The claim follows.
\end{proof}

\subsection{The left and right dual of a simple two-sided vector space}
We now works towards establishing formulas for the left and right dual of a simple two-sided vector space
(Theorem \ref{thm.dualformula}).  To this end, we will use the following notation and conventions for the remainder of the section.  We will routinely utilize the notation defined in (\ref{eqn.simpledata}) and (\ref{eqn.phi}) at the beginning of Section \ref{sect.babysimple}.  We assume $K$ is perfect and $\lambda \in \operatorname{Emb}(K)$ with $|\lambda^{G}|=n$, where, we remind the reader,  $G:=\operatorname{Aut}(\overline{K}/K)$.  In addition, we let $\overline{\lambda}$ be an extension of $\lambda$ to $\overline{K}$, we let $\overline{\mu}$ be the inverse of $\overline{\lambda}$ and we let $\mu=\overline{\mu}|_{K}$.  Finally, if $\gamma, \delta:K \rightarrow L$ are $k$-linear
embeddings of $K$ into a field $L$, we let ${}_{\gamma}L_{\delta}$ denote the two-sided vector space whose underlying set is $L$ and whose $K \otimes_{k}K$-module action is induced by the formula $(a \otimes b) \cdot c := \gamma(a)\delta(b) c$.

\begin{lemma} \label{lemma.newform}
The map $\overline{\lambda}$ induces an isomorphism ${}_{\mu}\mu(K) \vee K_{\operatorname{id}} \rightarrow V(\lambda)$.
\end{lemma}

\begin{proof}
The restriction of $\overline{\lambda}$ to $\mu(K) \vee K$ maps into $K(\lambda)$.
It follows that this restriction induces a map of two-sided vector spaces
${}_{\mu}\mu(K) \vee K_{\operatorname{id}} \rightarrow V(\lambda)$.  Since the kernel of $\overline{\lambda}$ is $0$,
and since the image of $\overline{\lambda}$ restricted to $\mu(K) \vee K$ is a subfield of $\overline{K}$ containing
$K$ and $\lambda(K)$, $\overline{\lambda}$ induces an isomorphism of two-sided vector spaces ${}_{\mu}\mu(K) \vee K_{\operatorname{id}} \rightarrow V(\lambda)$,
as desired.
\end{proof}
We next introduce notation which will be employed for the rest of this section.  Since $|\lambda^{G}|=n$, $\lambda^{G}=\{\sigma_{1}\lambda,\ldots,\sigma_{n}\lambda\}$ where $\sigma_{i} \in G$.  Without loss of generality, we assume $\sigma_{1}$ is the identity map.
\begin{lemma} \label{lemma.ext}
The $n \times n$ matrix whose $ij$ entry is $\sigma_{i}(\alpha_{j})$ is invertible.  Therefore, the function $\lambda_{i}:K \rightarrow K$ defined in (\ref{eqn.simpledata}) is a $\overline{K}$-linear combination of
elements of $\lambda^{G}$.
\end{lemma}

\begin{proof}
We have $\sigma_{i} \lambda = \sum_{j=1}^{n} \sigma_{i}(\alpha_{j})\lambda_{j}$ since
$\sigma_{i} \in \operatorname{Aut }(\overline{K}/K)$ and $\lambda_{j}:K \rightarrow K$.
The second assertion will thus follow from the first.  To prove the first assertion, we show that the columns of the matrix whose $ij$ entry is $\sigma_{j}(\alpha_{i})$ are linearly independent over $\overline{K}$.  Otherwise, without loss of generality,
\begin{equation} \label{eqn.columns}
\begin{pmatrix} \sigma_{n}(\alpha_{1}) \\ \vdots \\ \sigma_{n}(\alpha_{n}) \end{pmatrix} = a_{1} \begin{pmatrix} \sigma_{1}(\alpha_{1}) \\ \vdots \\ \sigma_{1}(\alpha_{n}) \end{pmatrix} + \cdots + a_{n-1}\begin{pmatrix} \sigma_{n-1}(\alpha_{1}) \\ \vdots \\ \sigma_{n-1}(\alpha_{n}) \end{pmatrix}
\end{equation}
where $a_{i} \in \overline{K}$.  Since $\sigma_{i}$ is $K$-linear, it induces a $K$-linear embedding of $K(\lambda)$ into $\overline{K}$.  Thus, (\ref{eqn.columns}) implies that $\sigma_{n}|_{K(\lambda)}=a_{1}\sigma_{1}|_{K(\lambda)}+ \cdots+a_{n-1}\sigma_{n-1}|_{K(\lambda)}$
which contradicts independence of distinct characters.  The lemma follows.
\end{proof}
We now introduce notation we will require in the statement of Lemma \ref{lemma.mult} and in the proof of Theorem \ref{thm.dualformula}.  By Lemma \ref{lemma.ext}, there is an inverse, $(a_{ij})$ to $(\sigma_{i}(\alpha_{j}))$.  Since, as we observed in the first line of the proof of Lemma \ref{lemma.ext}, we have $\sigma_{i} \lambda = \sum_{j=1}^{n} \sigma_{i}(\alpha_{j})\lambda_{j}$, it follows that $\lambda_{i}=\sum_{j}a_{ij}\sigma_{j}\lambda$.  We define $\overline{\lambda}_{i} =\sum_{j}a_{ij}\sigma_{j}\overline{\lambda}$.

\begin{lemma} \label{lemma.mult}
We have
$$
\overline{\lambda} = \sum_{i}{\overline{\lambda}}_{i}\alpha_{i}
$$
and, for $b, c \in \overline{K}$ such that $\overline{\lambda}_{i}(b), \overline{\lambda}_{i}(c)$, and $\overline{\lambda}_{i}(bc)$ are in $K$,
$$
{\overline{\lambda}}_{k}(bc)=\sum_{i,j}\overline{\lambda}_{i}(b)\overline{\lambda}_{j}(c)\beta_{ijk}.
$$
\end{lemma}

\begin{proof}
It follows from the fact that $(a_{ij})$ and $(\sigma_{i}(\alpha_{j}))$ are inverse that $\sigma_{i}\overline{\lambda}=\sum_j \sigma_{i}(\alpha_{j}){\overline{\lambda}}_{j}$.  When $i=1$, we get the first assertion.

To prove the second assertion, we note that since $\overline{\lambda}(bc)=\overline{\lambda}(b)\overline{\lambda}(c)$, we have
$$
\sum_{i}\overline{\lambda}_{i}(bc)\alpha_{i} =
(\sum_{i}\overline{\lambda}_{i}(b)\alpha_{i})(\sum_{j} \overline{\lambda}_{j}(c)\alpha_{j}).
$$
It follows that $\overline{\lambda}_{k}(bc)$ is equal to the coefficient of $\alpha_{k}$ in
$\sum_{i,j}\overline{\lambda}_{i}(b)\overline{\lambda}_{j}(c)\alpha_{i}\alpha_{j} =
\sum_{i,j,k}\overline{\lambda}_{i}(b)\overline{\lambda}_{j}(c)\beta_{ijk}\alpha_{k}$.  The result follows.
\end{proof}

\begin{thm} \label{thm.dualformula}
Suppose $[K(\lambda):\lambda(K)]$ is finite.  Then ${}^{*}V(\lambda) \cong V(\lambda)^{*} \cong V(\mu)$, and $V(\lambda)$ has rank $n$ if and only if $V(\mu)$ has rank $n$ if and
only if $[K(\lambda):\lambda(K)]=n$.
\end{thm}

\begin{proof}
We begin by noting that since $[K(\lambda):\lambda(K)]=m < \infty$, we must have $[K(\mu):K]=m$, so that the notation $V(\mu)$ makes sense and the module $V(\mu)$ is simple.  We now prove that $V(\lambda)^{*} \cong V(\mu)$.  In order to proceed, we need to make a remark regarding notation.  We adopt the notation defined in (\ref{eqn.simpledata}) and (\ref{eqn.phi}) at the beginning of Section \ref{sect.babysimple}, using $\mu$ in
place of $\lambda$.  Furthermore, we define $\overline{\mu}_{i}$ in the same way we defined $\overline{\lambda}_{i}$ preceding the statement of Lemma \ref{lemma.mult}.  We will show that
$\{\overline{\mu}_{1},\ldots, \overline{\mu}_{m}\}$ is a left basis for $V(\lambda)^{*}$, and in this basis the
matrix for the right action of $K$ is the matrix $\phi$ defined by (\ref{eqn.phi}).  It will follow that $V(\lambda)^{*} \cong V(\mu)$.

We begin by showing that $\overline{\mu}_{i}$ restricts to an element of $\operatorname{Hom}_{K}(V(\lambda)_{K},K)$.
By the definition of $(a_{ij})$ and $\overline{\mu}_{i}$ following the proof of Lemma \ref{lemma.ext}, if $c_{l},d_{l} \in K$ where $l$ runs over a finite index set, then
$\overline{\mu}_{i}(\sum_{l}c_{l}\lambda(d_{l})) \in K$.  Since $[K(\lambda):K]$ is finite, it follows that
$\overline{\mu}_{i}$ restricted to $V(\lambda)$ takes values in $K$.  If $\alpha \in K(\lambda)$ and $b \in K$, then
\begin{eqnarray*}
\overline{\mu}_{i}(\alpha \cdot b) & = & \overline{\mu}_{i}(\lambda(b) \alpha) \\
& = & \sum_{j}a_{j}\sigma_{j}\overline{\mu}(\lambda(b)) \sigma_{j}\overline{\mu}(\alpha) \\
& = & b \overline{\mu}_{i}(\alpha) \\
& = & \overline{\mu}_{i}(\alpha) \cdot b.
\end{eqnarray*}
It follows that $\overline{\mu}_{i} \in \operatorname{Hom}_{K}(V(\lambda)_{K},K)$ as desired.

We next show that $\{\overline{\mu}_{1},\ldots, \overline{\mu}_{m}\}$ is a left basis for $V(\lambda)^{*}$.  Since $V(\lambda)$ has right dimension $m$ by Proposition \ref{prop.vlambda}, $V(\lambda)^{*}$ has left dimension $m$.  Therefore, it suffices to show that $\{\overline{\mu}_{1},\ldots, \overline{\mu}_{m}\}$ is left linearly independent.  To this end, since $\{\sigma_{1}\overline{\mu},\ldots,\sigma_{m} \overline{\mu}\}$ is a set of distinct embeddings, this set is left linearly independent over $\overline{K}$.  Since the matrix  $(a_{ij})$ defined in the paragraph following the proof of Lemma \ref{lemma.ext} and the matrix $(\sigma_{i}(\alpha_{j}))$ are inverse, $\sigma_{i} \overline{\mu} = \sum_{j=1}^{m} \sigma_{i}(\alpha_{j})\overline{\mu}_{j}$.  It follows that $\{\overline{\mu}_{1},\ldots, \overline{\mu}_{m}\}$ is left linearly independent over $\overline{K}$, hence over $K$.  Therefore, $\{\overline{\mu}_{1},\ldots, \overline{\mu}_{m}\}$ is
a left basis for $V(\lambda)^{*}$.

We complete the proof that $V(\lambda)^{*} \cong V(\mu)$ by computing the matrix for the right action of $K$ on $V(\lambda)^{*}$ in the left basis $\{\overline{\mu}_{1},\ldots, \overline{\mu}_{m}\}$.  If $b \in K$, then the argument in the second paragraph of the proof implies that $\overline{\mu}_{k}(b), \overline{\mu}_{k}(\alpha_{l})$, and $\overline{\mu}_{k}(b\alpha_{l})$ are elements of $K$.  Therefore, Lemma \ref{lemma.mult} implies that $\overline{\mu}_{k}(b\alpha_{l})$ equals $\sum_{i,j}\overline{\mu}_{i}(b)
\overline{\mu}_{j}(\alpha_{l})\beta_{ijk}$.  This, in turn equals
$$
\sum_{j}[(\sum_{i}\overline{\mu}_{i}(b)\beta_{ijk})\overline{\mu}_{j}(\alpha_{l})]=
\sum_{j}[(\sum_{i}\overline{\mu}_{i}(b)\beta_{jik})\overline{\mu}_{j}(\alpha_{l})] .
$$
Thus, the $kj$ entry of the matrix for the right action is $\sum_{i}\beta_{jik}\mu_{i}(b)$ since $b \in K$.  It follows that the matrix for the right action is the matrix $\phi$ (defined in (\ref{eqn.phi})) corresponding to $\mu$.

To complete the proof of the first part of the theorem, we need to prove that ${}^{*}V(\lambda) \cong V(\mu)$.  Since
the proof is similar to the proof that $V(\lambda)^{*} \cong V(\mu)$, we only sketch it.  By
Lemma \ref{lemma.newform}, if $W = {}_{\mu}\mu(K) \vee K_{\operatorname{id}}$, then
${}^{*}V(\lambda) \cong {}^{*}W$.  It is straightforward to check, as above, that
 $\{\overline{\lambda}_{1},\ldots, \overline{\lambda}_{n}\}$ is a right basis for ${}^{*}W$, which allows
 one to prove that ${}^{*}W \cong {}_{\lambda}\lambda(K) \vee K_{\operatorname{id}}$.  Thus, by Lemma
 \ref{lemma.newform}, ${}^{*}W \cong V(\mu)$.

Finally, we prove the second part of the theorem.  By Theorem \ref{thm.papp} and Proposition \ref{prop.vlambda}, $n=[K(\lambda):K]$.
Furthermore, by Lemma \ref{lemma.newform}, $[K(\mu):K]=[K(\lambda):\lambda(K)]$.  Similarly, $[K(\mu):\mu(K)]=[K(\lambda):K]$.
It follows from Proposition \ref{prop.vlambda} that $V(\lambda)$ has
rank $n$ if and only if $V(\mu)$ has rank $n$ if and only if $[K(\lambda):\lambda(K)]=n$.
\end{proof}

\section{Non-commutative symmetric algebras} \label{sec.ncsym}
In this section we recall (from \cite{p1bundle}) the definition of the non-commutative symmetric algebra of a rank $n$ two-sided vector space $V$.  We conclude the paper by using Theorem \ref{thm.dualformula} to show that the non-commutative symmetric algebra of $V$ exists over a perfect field.

\begin{defn} \label{defn.ncsymalg}
Let $V$ be a rank $n$ two-sided vector space over $K$ such that $V^{i*}$ has rank $n$ for all $i$.  Since $(-\otimes_{K}V^{i*},-\otimes_{K}V^{(i+1)*})$
is an adjoint pair for each $i$ by Proposition \ref{prop.adjoint} and Corollary \ref{cor.matrix0}, the Eilenberg-Watts Theorem implies that the unit of the pair $(-\otimes_{K}V^{i*},-\otimes_{K}V^{(i+1)*})$ induces a map of two-sided vector spaces
$K \rightarrow V^{i*} \otimes_{K} V^{(i+1)*}$.  We denote the image of this map by $Q_{i}$.

The {\it non-commutative symmetric algebra generated by $V$},
denoted $\mathcal{S}_{K}^{n.c.}(V)$, is the $\mathbb{Z}$-algebra (see \cite[p. 95]{pol} for a definition of $\mathbb{Z}$-algebra) $\underset{i,j \in
\mathbb{Z}}{\oplus}A_{ij}$ with components defined as follows:
\begin{itemize}
\item{} $A_{ij}=0$ if $i>j$.

\item{} $A_{ii}=K$.

\item{} $A_{i,i+1}=V^{i*}$.
\end{itemize}
In order to define $A_{ij}$ for $j>i+1$, we introduce some notation: we define $B_{i,i+1}=A_{i,i+1}$, and, for $j>i+1$, we define
$$
B_{ij}=A_{i,i+1} \otimes A_{i+1,i+2} \otimes \cdots \otimes A_{j-1,j}.
$$
We let $R_{i,i+1}=0$, $R_{i,i+2}=Q_{i}$,
$$
R_{i,i+3}=Q_{i} \otimes V^{(i+2)*}+V^{i*} \otimes Q_{i+1},
$$
and, for $j>i+3$, we let
$$
R_{ij}=Q_{i} \otimes B_{i+2,j}+B_{i,i+1}\otimes Q_{i+1} \otimes B_{i+3,j}+\cdots + B_{i,j-2} \otimes Q_{j-2}.
$$

\begin{itemize}
\item{} For $j>i+1$, we define $A_{ij}$ as the quotient $B_{ij}/R_{ij}$.
\end{itemize}
Multiplication in $\mathcal{S}_{K}^{n.c.}(V)$ is defined as follows:
\begin{itemize}

\item{} if $a \in A_{ij}$, $b \in A_{lk}$ and $j \neq l$, then $ab=0$,

\item{} if $a \in A_{ij}$ and $b \in A_{jk}$, with either $i=j$ or $j=k$, then $ab$ is induced by the usual scalar action,

\item{}  otherwise, if $i<j<k$, we have
\begin{eqnarray*}
A_{ij} \otimes A_{jk} & = & \frac{B_{ij}}{R_{ij}} \otimes \frac{B_{jk}}{R_{jk}}\\
& \cong & \frac{B_{ik}}{R_{ij}\otimes B_{jk}+B_{ij} \otimes
R_{jk}}.
\end{eqnarray*}
Since $R_{ij} \otimes B_{jk}+B_{ij} \otimes R_{jk}$ is a submodule of $R_{ik}$,
there is thus an epi $\mu_{ijk}:A_{ij} \otimes A_{jk} \rightarrow
A_{ik}$.
\end{itemize}
We say the non-commutative symmetric algebra of a rank $n$ vector space $V$ {\it exists} if $V^{i*}$ has rank $n$ for all $i \in \mathbb{Z}$.
\end{defn}
It is straightforward to check, using \cite[Lemma 6.6]{duality}, that if $\mathcal{S}_{K}^{n.c.}(V)$ exists and $f:V \rightarrow W$ is an isomorphism,
then $f$ induces an isomorphism of $\mathbb{Z}$-algebras $\mathcal{S}_{K}^{n.c.}(V) \cong \mathcal{S}_{K}^{n.c.}(W)$.

\begin{lemma} \label{lemma.dual}
Suppose $K$ is a perfect field and $W$ is simple with $\operatorname{dim }_{K}({}_{K}W)=n<\infty$ and $\operatorname{dim }_{K}(W_{K})=m < \infty$.  If $i$ is even, then $\operatorname{dim }_{K}({}_{K}W^{i*})=n$ and $\operatorname{dim }_{K}({W^{i*}}_{K})=m$.  If $i$ is odd, then $\operatorname{dim }_{K}({}_{K}W^{i*})=m$ and $\operatorname{dim }_{K}({W^{i*}}_{K})=n$.
\end{lemma}

\begin{proof}
Since $W$ is simple, Proposition \ref{prop.vlambda} implies there exists $\lambda \in \operatorname{Emb}(K)$ such that $W \cong V(\lambda)$.  Theorem \ref{thm.dualformula} and Corollary \ref{cor.matrix0} thus imply that for $i$ even, $W^{i*} \cong V(\lambda)$ and for $i$ odd, $W^{i*} \cong V(\mu)$.  The result follows from Proposition \ref{prop.vlambda} since $[K(\lambda):K]=[K(\mu):\mu(K)]$ and $[K(\lambda):\lambda(K)]=[K(\mu):K]$.
\end{proof}

\begin{prop} \label{prop.dual}
Suppose $K$ is perfect and $V$ is a two-sided vector space over $K$ with $\operatorname{dim }_{K}({}_{K}V)=n<\infty$ and $\operatorname{dim }_{K}(V_{K})=m < \infty$.  If $i$ is even, then $\operatorname{dim }_{K}({}_{K}V^{i*})=n$ and $\operatorname{dim }_{K}({V^{i*}}_{K})=m$.  If $i$ is odd, then $\operatorname{dim }_{K}({}_{K}V^{i*})=m$ and $\operatorname{dim }_{K}({V^{i*}}_{K})=n$.
\end{prop}

\begin{proof}
We prove the result when $i=1$.  When $i>1$ the result follows from induction on $i$ since $(V^{i*})^{*}=V^{(i+1)*}$.  If $i<0$, a similar argument works, so we omit the proof in this case.  To prove the result in the $i=1$ case, we proceed by induction on the left dimension, $n$, of $V$.  If $n=1$ then $V$ is simple, so the result follows from Lemma \ref{lemma.dual}.  In the general case, if $V$ is simple, then the result follows from Lemma \ref{lemma.dual}.  Otherwise, let $W \subset V$ be simple and suppose $\operatorname{dim }_{K}({}_{K}W)=p$.  By Lemma \ref{lemma.functorial}, there is an exact sequence
$$
0 \rightarrow (V/W)^{*} \rightarrow V^{*} \rightarrow W^{*} \rightarrow 0.
$$
By Lemma \ref{lemma.dual}, $\operatorname{dim }_{K}({V^{*}}_{K})=p+\operatorname{dim }_{K}({(V/W)^{*}}_{K})$.  By the induction hypothesis, $\operatorname{dim }_{K}({(V/W)^{*}}_{K})=n-p$.  It follows that $\operatorname{dim }_{K}({V^{*}}_{K})=n$ as desired.  The proof in case $i=1$ follows.
\end{proof}

\begin{cor} \label{cor.ncsym}
If $K$ is a perfect field and $V$ has rank $n$ over $K$, then $\mathcal{S}_{K}^{n.c.}(V)$ exists.
\end{cor}

\begin{proof}
It suffices to show that $V^{i*}$ has rank $n$ for all $i \in \mathbb{Z}$.  This follows immediately from Proposition \ref{prop.dual}.
\end{proof}

\end{document}